\documentclass[12pt]{amsart}
\usepackage[left=2.5cm,top=2.5cm,right=2.5cm]{geometry}
\usepackage{verbatim}
\usepackage{amsmath}
\usepackage{amssymb}
\usepackage{MnSymbol}
\usepackage[all,cmtip]{xy}
\usepackage{color}
\usepackage{url}
\usepackage{graphicx}

\numberwithin{equation}{section}

\newtheorem{theorem}[equation]{Theorem}
\newtheorem{lemma}[equation]{Lemma}

\newtheorem{proposition}[equation]{Proposition}

\theoremstyle{remark}
\newtheorem{remark}[equation]{Remark}

\renewcommand{\O}{{\mathcal O}}

\def\AA{{\mathbb A}}
\def\CC{{\mathbb C}}
\def\QQ{{\mathbb Q}}

\def\FF{{\mathbb F}}

\def\ZZ{{\mathbb Z}}

\def\II{{\mathbb I}}

\def\QQl{\bar{\QQ}_l}
\def\BN{\mathcal{B}_N}
\def\tBN{\tilde{\mathcal{B}}_N}
\def\TN{T_N}
\def\TN{\tilde{T}_N}

\newcommand{\tG}{{\tilde{G}}}
\newcommand{\tf}{{\tilde{f}}}

\newcommand{\tpi}{{\tilde{\pi}}}
\newcommand{\tnu}{{\tilde{\nu}}}
\newcommand{\tlambda}{{\tilde{\lambda}}}

\newcommand{\res}{{\text{res}_{G}^{\tilde{G}}}}

\newcommand{\tZ}{{\tilde{Z}}}
\newcommand{\tK}{{\tilde{K}}}

\newcommand{\tT}{{\tilde{T}}}
\newcommand{\tc}{{\tilde{c}}}
\newcommand{\tk}{{\tilde{k}}}

\newcommand{\oI}{{I_{\omega}}}
\newcommand{\tmu}{{\tilde{\mu}}}

\newcommand{\tcc}{{(\tilde{Z}(\mathbb{A})/\tilde{Z}(F))^D}}
\newcommand{\cc}{{(Z(\mathbb{A})/Z(F))^D}}

\newcommand{\cusp}{{A_0}}
\newcommand{\pcusp}{{\Pi_{\text{cusp}}}}
\newcommand{\Fu}{{F^{\ast}}}
\newcommand{\tTheta}{{\tilde{\Theta}}}
\newcommand{\tsigma}{{\tilde{\sigma}}}

\begin{document}

\title{A note on the multiplicity of $SL(n)$ over function fields}

\date{\today}

\author{Yang An}

\begin{abstract}

In \cite{lafforgue2012chtoucas}, Vicent Lafforgue attaches a semisimple Langlands parameter (or, what amounts to the same thing, a $\hat{G}$-pseudocharacter) to every cuspidal automorphic representation of a reductive group $G$ over the field of functions of a smooth projective algebraic curve $X$ over a finite field. Hence, gets a decomposition of the space of cusp forms. In this note, we show that in the case of $G = SL(n)$, Lafforgue's decomposition coincides with the classical decomposition using $L$-packets, and moreover, the number of ($G$-equivalence classes of) extensions of an unramified Hecke character of $G$ to $\hat{G}$-pseudocharacters serves as a natural upper bound on the multiplicity of $SL(n)$.

\end{abstract}

\maketitle

\section{Overview}

Let $k = \FF_q$, where $q=p^m$ is a prime power, and $X/k$ be a smooth irreducible projective curve and $F=k(X)$ its field of rational functions. The set of places of $F$ is denoted as $\vert X \vert$, which is the same thing as the set of closed points of $X$. For each $v \in |X|$, we have $F_v$ its completion, and $\mathcal{O}_v$ its ring of integer in $F_v$. Let $G$ be a connected reductive group over $F$, and $N = \sum_{v\in |X|} n_v v$ be an effective divisor on $X$, and $K_N = \{k \in \Pi_{v\in |X|} G(\mathcal{O}_v): k \equiv 1 (\text{mod } m_v^{n_v})\}$ be the open compact subgroup of level $N$.

Fix some prime number $l \neq p$, let $A_0(G,\bar{\QQ}_l)$ denote the space of cusp forms, and $A_0(G,K_N,\bar{\QQ}_l) = A_0(G,\bar{\QQ}_l)^{K_N}$. In \cite{lafforgue2012chtoucas}, V. Lafforgue construct a commutative algebra $\mathcal{B}_N$, containing the normal Hecke algebra $T_N = \bigotimes_{v \nmid N} T_v$, called the excursion algebra(of level $N$). Moreover, for each excursion character $\nu: \mathcal{B}_N \rightarrow \bar{\QQ}_l$, one can associate a unique Langlands parameter $\sigma_{\nu}: \Gamma \rightarrow \hat{G}(\QQl)$, up to conjugation, where $\Gamma$ is the Galois group of the maximal separable extension of $F$ unramified outside of the support of $N$, and $\hat{G}$ is the Langlands dual group of $G$. More precisely, $\BN$ is generated by excursion operators $S_{m,f,\gamma} \in End(\cusp(G,K_N,\QQl))$, where $f \in \mathcal{O}(\hat{G}^{n})^{\hat{G}}$, the conjugate invariant functions on $\hat{G}^{n}$, and $\gamma = (\gamma_1, \ldots \gamma_n)\in \Gamma^n$. When $V \in Rep(\hat{G})$, the Hecke operator $h_{V,v}$ is $S_{1,Tr_V,Frob_v}$. Moreover, if $\nu$ is a character of $\BN$, then the function $\Theta(m,f,\gamma) = \nu(S_{m,f,\gamma})$ is called a pseudocharacter of $G$. It is the pseudocharacter that gives a unique Langlands parameter $\sigma$, up to conjugacy, they are related by:
$$\Theta(m,f,\gamma) = \nu(S_{m,f,\gamma}) = f(\sigma(\gamma)).$$

Then V. Lafforgue proved a decomposition\cite{lafforgue2012chtoucas} :

\begin{equation}\label{eqn: Lafforgue decomposition}
  A_0(G,K_N,\bar{\QQ}_l) = \bigoplus_{\nu} A_{0,\nu} = \bigoplus_{\sigma} A_{0,\sigma}.
\end{equation}
Moreover, by the identification of excursion operators $S_{1,Tr_V,Frob_v}$ with the Hecke operators, it is easy to see this decomposition is also compatible with the Satake isomorphism, more precisely, for any $\nu: \BN \rightarrow \QQl$, and $v \nmid N$, then $\sigma_{\nu}$ is unramified at $v$, and the semisimple conjugacy class $\sigma_{\nu}(Frob_v)$ corresponds to the Hecke character $\sigma_{\nu} \vert_{T_v}$ under the Satake isomorphism.

On the other hand, we know that $A_0(G,\QQl)$ is a discrete $G(\AA)$ module, and has a decomposition:
\begin{equation} \label{eqn:classical decomposition}
  A_0(G,\QQl) = \bigoplus_{\pi} m(\pi) \pi,
\end{equation}
with pairwise inequivalent irreducible admissible $G(\AA)$ representations, and $m(\pi)$ is a finite nonnegative integer. After taking the $K_N$ fixed part, decomposition (\ref{eqn:classical decomposition}) just becomes:

\begin{equation}\label{eqn:leveled classical decomposition}
  A_0(G,K_N,\QQl) = \bigoplus_{\pi} m(\pi^{K_N}) \pi^{K_N},
\end{equation}

While it is well-known that $m(\pi) = 1$ or $0$ in the case of $G = GL(n)$, in \cite{blasius1994multiplicities}, Blasius constructed infinitely many families of automorphic cuspidal representations that are isomorphic, but not coincide in the case of $G=SL(n)$ over number fields, namely $m(\pi)>1$ for infinitely many $\pi$. This raises the question of higher multiplicities of $SL(n)$. We studied Lafforgue's excursion character in the case of $G = SL(n)$ over function fields, and see how the Lafforgue's decomposition (\ref{eqn: Lafforgue decomposition}), and the idea of pseudocharacters account for the multiplicities of $SL(n)$.

After showing that Lafforgue's decomposition (\ref{eqn: Lafforgue decomposition}) coincides with the classical decomposition by $L$-packets induced from $GL(n)$ in Subsection \ref{subsec: coincide}, we come to Proposition \ref{prop:multiplcity of SLn}, which gives an uppoer bound of multiplicities of $SL(n)$ in terms of the number of extensions:

\begin{proposition}\label{prop:intro multiplicity}
  The number of isomorphic irreducible components of $\cusp(G,\overline{\QQ_l})$ (which corresponds to a character $\lambda$ of some unramified Hecke algebra of $SL(n)$), is bounded above by the number of $G$-equivalent classes of pseudocharacters $\tTheta(m,f,(\gamma_i))$ of GL$(n)$, such that $\tTheta(1,Tr_V,\gamma)$ (where $V$ is any representation of $GL(n)$ that factors through $PGL(n)$) is given by $\lambda$, and $\tTheta(GL(n))(1,Det,\gamma)$ is given by $\tmu$.
\end{proposition}
And we will see that the $\lambda$ and $\tmu$ in Proposition \ref{prop:intro multiplicity} determine pseudocharacters $\tTheta$ up to $n$-th roots of unity.

In the rest of this paper, we will fix the following notations, $n$ is a positive integer coprime to $p$, and  $\tilde{G} = GL(n)$, $G = SL(n)$. All notations with $\tilde{}$ will refer to the correponding notion of $GL(n)$, and those without $\tilde{}$ will refer to those of $SL(n)$. For example, $\tilde{T}_N$ will indicate the Hecke algebra (of level $N$) of $GL(n)$, and $\pi$ will indicate a cuspidal representation of $SL(n)$, i.e. $\pi \subset A_0(SL(n), \QQl)$, unless otherwise specified.\\

\textbf{Acknowledgement.} I would like to express my sincere gratitudes to my advisor Michael Harris, for suggesting this problem to me, and many help and encouragement throughout. I would also like to thank Hang Xue for his helpful suggestions.

\section{Restriction of cusp forms}\label{sec:global}

In this section, we review classical theory about the restriction of cusp forms from $GL(n)$ to $SL(n)$, our main reference is \cite{hiraga2012}. Although the theorems proved in \cite{hiraga2012} is for number fields, but the proofs work verbatim in the function field case, provided our standing assumption that $(n,\text{Char}(F)) = 1$. Denote $\tZ$ to be the center for $\tG$, and $Z=\tZ\cap G$. Let $\tmu \in \tcc$, the Pontryagin dual of $\tZ(\AA)/\tZ(F)$, and $\mu \in \cc$. We will suppress the coefficient $\QQl$, and write $\cusp(\tG,\tmu)$ to indicate the cusp forms with central character $\tmu$. We denote $\pcusp(\tG,\tmu)$ to be the set of irreducible representations of $\tG(\AA)$ appearing in $\cusp(\tG,\tmu)$. Let $\mathcal{L}(\tpi)$ be the maximal $\tpi$-isotypic subspace of $\cusp(\tG,\tmu)$. Then $\mathcal{L}(\tpi)=m(\tpi)\tpi$, where $m(\tpi)=1$ by Multiplicity One Theorem for $GL_n$. Similarly, we define $\cusp(G,\mu)$, $\pcusp(G,\mu)$, and $\mathcal{L}(\pi)$ for $\pi \in \pcusp(G,\mu)$ for $SL(n)$.

Before we go into the global theory of restriction of cusp forms, we need to make a little preparations in the local case.

\subsection{Restriction in the local case}
We first make a little general assumption in the local case, suppose $\tG$ is a locally compact, totally disconnected group, i.e. it has a fundamental basis of neighbourhood around $e$ by open compact subgroups. Let $G \subset \tG$ to be an open normal subgroup of finite index.

The following lemmas are carefully proved in \cite{gelbart1982indistinguishability},

\begin{lemma}\label{gk1}
  If $\tpi$ is an irreducible admissible representation of $\tG$, then $\text{Res}_{G}^{\tG}(\tpi)$ is a direct sum of finite number of irreducible admissible representations of $G$ with the same multiplicity.
\end{lemma}

We denote by $\Pi(\tpi)=\Pi_G^{\tG}(\tpi)$ the set of equivalence classes of irreducible admissible representations of $G$ appearing in the composition series of $\text{Res}_{G}^{\tG}(\tpi)$. Then the above lemma asserts that

\begin{equation}\label{eqn:localdecomp}
  \text{Res}_G^{\tG}(\tpi)=\bigoplus_{\pi \in \Pi(\tpi)} m\cdot \pi,
\end{equation}

where $m$ is the common multiplicity of $\pi \in \Pi(\tpi)$.

\begin{lemma}\label{gk2}
  Let $\tpi$ and $\tpi'$ be irreducible admissible representations of $\tG$, then the following are equivalent:
  \begin{itemize}
    \item $\Pi(\tpi) \cap \Pi(\tpi') \neq \emptyset$,
    \item $\Pi(\tpi) = \Pi(\tpi')$,
    \item $\tpi'\cong \tpi \otimes \omega$ for some $\omega \in (\tG/G)^D$.
  \end{itemize}
\end{lemma}

Now, we specialize to our interesting case. For any $v\in |X|$, let $\tG$ be $\tG_v=GL(n,F_v)$, and similarly $G$ be $G_v=SL(n,F_v)$, and $\tpi_v$ is a component of an irreducible cuspidal representation $\tpi$, with $(n,\text{Char}(F))=1$. Suppose $\tpi$ has a central character, i.e. there exists a character $\tmu$ of $\Fu$ such that $\tpi(x)=\tmu(x)\cdot \text{Id}$, where $x$ is identified with its diagonal embedding of $\Fu$ into $\tG_v$. Let $H=\tZ\cdot G$ be a subgroup of $\tG$, we then have an isomorphism (as a topological group) $\tG/G \cong \Fu$ through determinant map, where $H/G$ is the inverse image of $\Fu^n$, the group of $n$-th power. Since $(n,\text{Char}(F))=1$, we know that $\Fu^n$ is an open normal subgroup of $\Fu$ of finite index. Hence, the same holds for $H\subset \tG$.

Since $H=\tZ G$, where $\tpi_v$ is just the scalar when restricted to $\tZ$. We know that a representation of $\tZ G$ is irreducible if and only if its restriction on $G$ is irreducible, hence Lemma \ref{gk1} and Lemma \ref{gk2} still holds for our $\tG_v$ and $G_v$. Finally, we note the following theorem of Tedic \cite{tadic1992notes}(Theorem 1.2):
\begin{theorem}\label{thm:tadic}
  For an irreducible smooth representation $(\tpi_v,V)$ of $\tG_v$, $\tpi_v \vert_{G_v}$ is multiplicity free.
\end{theorem}

\subsection{(Locally) $G$-equivalence}

For $\mu \in (Z(\AA)/Z(F))^D$, we put
\begin{equation*}
  [\mu]=\{\tmu \in (\tZ(\AA)/\tZ(F))^D \vert \text{Res}^{\tZ(\AA)}_{Z(\AA)}\tmu=\mu\}.
\end{equation*}
For $\tmu,\tmu' \in (\tZ(\AA)/\tZ(F))^D$, we say that $\tmu$ and $\tmu'$ are $G$-equivalent if there exists $\omega \in (\tG(\AA)/\tG(F)G(\AA))^D$ such that \begin{equation*}
  \tmu'\cong\tmu \otimes \text{Res}^{\tG(\AA)}_{\tZ(\AA)}\omega.
\end{equation*}
It can be shown that $\tmu,\tmu'$ are $G$-equivalent, if and only if they are in the same $[\mu]$.

For $\tf\in \cusp (\tG, \tmu)$, we define $\res(\tf)$ to be the restriction of $\tf$ to a function on  $G(\AA)$. For $\tmu \in [\mu]$, we clearly have $\res\cusp(\tG,\tmu)\subset \cusp(G,\mu)$.
\begin{remark}
  This is different from Res, which we reserved for restriction as abstract representation.
\end{remark}

Let $\tpi \in \pcusp(\tG, \tmu)$, and $\tpi' \in \pcusp(\tG, \tmu')$. We say that $\tpi'$ and $\tpi$ are \textbf{locally $G$-equivalent} if, for any place $v$ of $F$, there exists $\omega_v \in (\tG(F_v)/G(F_v))^D$ such that $\tpi'_v\cong \tpi_v \otimes \omega_v$ (or equivalently, there exists $\omega \in (\tG(\AA)/G(\AA))^D$ such that $\tpi' \cong \tpi \otimes \omega$). Furthermore, $\tpi'\in\pcusp(\tG,\tmu')$ and $\tpi \in \pcusp(\tG,\tmu)$ are \textbf{$G$-equivalent}, if there exists $\omega \in (\tG(\AA)/G(\AA)\tG(F))^D$, such that $\tpi'\cong \tpi \otimes \omega$. If $\tpi'$ and $\tpi$ are $G$-equivalent, then $\tmu$ and $\tmu'$ are $G$-equivalent. If $\tpi,\tpi'\in \pcusp(\tG,\tmu)$, then they are $G$-equivalent, if and only if the chosen $\omega \in (\tG(\AA)/G(\AA)\tZ(\AA)\tG(F))^D$. We denote the (resp. locally) $G$-equivalent class of $\tpi$ in $\pcusp(\tG,\tmu)$ to be $\{\tpi\}_G$(resp. $\{\tpi\}^{\text{loc}}_G$), and we write $\pcusp(\tG,\tmu)_G$ to be the set of $G$-equivalent classes in $\pcusp(\tG,\tmu)$.

 We put

\begin{equation*}
  X(\tpi)=\{\omega \in (\tG(\AA)/\tG(F)G(\AA)\tZ(\AA))^D \vert \tpi\otimes \omega \cong \tpi\},
\end{equation*}
\begin{equation*}
   X_{\text{loc}}(\tpi)=\{\omega \in (\tG(\AA)/G(\AA)\tZ(\AA))^D \vert \tpi\otimes \omega \cong \tpi\}.
\end{equation*}

For $\omega \in (\tG(\AA)/\tG(F)G(\AA))^D$, we define the twisting operator
\begin{equation*}
  \oI: \cusp(\tG,\tmu) \rightarrow \cusp(\tG, \tmu\otimes \text{Res}^{\tG(\AA)}_{\tZ(\AA)}\omega)
\end{equation*}
by $\oI\tf(x)=\omega(x)\tf(x)$. Then
\begin{equation*}
  \oI(\mathcal{L}(\tpi))=\mathcal{L}(\tpi\otimes\omega).
\end{equation*}
This implies $$\res\mathcal{L}(\tpi)=\res\mathcal{L}(\tpi\otimes\omega).$$
Note that we need $\omega \in (\tG(\AA)/\tG(F)G(\AA))^D$ rather than simply $\omega \in (\tG(\AA)/G(\AA))^D$. Otherwise $\omega(x)\tf(x)$ is not necessarily $G(F)$-invariant.

Let $S(\tpi)=\{\oI \vert \omega \in X(\tpi)\}$, then $S(\tpi)$ is commutative and acts on $\mathcal{L}(\tpi) \cong \tpi$. For $\eta \in X(\tpi)^D$, we put $$\mathcal{L}(\tpi)^\eta=\{\tf \in \mathcal{L}(\tpi) \vert \oI \tf = \eta(\omega) \tf, \text{for all} \quad \omega \in X(\tpi)\}.$$ Then $$\mathcal{L}(\tpi)=\bigoplus_{\eta \in X(\tpi)^D} \mathcal{L}(\tpi)^\eta.$$ The subspace $\mathcal{L}(\tpi)^1$ is where $S(\tpi)$ acts trivially. Since $\text{supp}(f)\subset \{ g\in \tG(\AA) \vert \omega(g)=\eta(\omega)\}$ for all $\omega \in X(\tpi)$. We see that

$$\res \mathcal{L}(\tpi)^\eta = 0, \quad \text{unless } \eta=1.$$
Thus, $\res \mathcal{L}(\tpi)=\res \mathcal{L}(\tpi)^1.$

The main theorem about the restriction of the cusp forms is the following:

\begin{theorem}\label{thm:main}
  For any $\tmu\in[\mu]$, the morphism
  \begin{equation}
\xymatrixcolsep{5pc}\xymatrix{
   \bigoplus_{\{\tpi\}_G\in\pcusp(\tG,\tmu)_G}\mathcal{L}(\tpi)^1  \ar[r]^\res & \cusp(G,\mu)}
  \end{equation}
is a bijection.
\end{theorem}

On the other hand, by Theorem \ref{thm:tadic}, we will get the following lemma about the restriction of irreducible cuspidal representations of $\tG$ to $G$ as abstract representations.

\begin{lemma}\label{lem:multiplicity}
  Let $\tpi \in \pcusp(\tG,\tmu)$, then Res$^\tG_G \tpi$ is a direct sum of irreducible admissible representations of $G$, and is multiplicity free.
\end{lemma}


We will summarize what we know about restrictions of cusp forms. For any cuspidal representation $\pi$ of $G(\AA)$, we put $$[\pi]=\bigcup _{\tmu \in [\mu]}\{\tpi \in \pcusp(\tG,\tmu)\vert \pi_v \in \Pi_{G_v}^{\tG_v}(\tpi_v), \text{for all places} \quad v\}.$$ By Lemma \ref{gk2}, we know $\tpi_1, \tpi_2 \in [\pi]$ is the same thing as $\tpi_1$ and $\tpi_2$ are locally $G$-equivalent, and their restrictions to $G$(as an abstract representation) has a constituent $\pi$.

For any $\tpi \in \pcusp(\tG,\tmu)$, we have $\res \tpi$ is a subrepresentation of $\cusp(G,\mu)$, and hence a direct sum of irreducible cuspidal representations of $G$. For each component $\pi$ in $\res \tpi$, Theorem \ref{thm:main} implies that $\res$ induces a $G$-isomorphism between $\pi$ and its preimage in $\tpi$. By Lemma \ref{lem:multiplicity}, we know that $\res \tpi$ is multiplicity free as well.

By Theorem \ref{thm:main}, we have $\cusp(G,\mu) = \bigoplus_{\{\tpi\}_G\in\pcusp(\tG,\tmu)_G} \res \tpi$. Moreover, if $\pi \in \pcusp(G,\mu)$ appears in both $\res \tpi_1$, and $\res \tpi_2$, then $\tpi_1, \tpi_2 \in [\pi]$. However, we don't know if the converse is true, i.e., if $\tpi \in [\pi]$, we can't say $\pi$ appears in $\res \tpi$ a priori.

In summary, there is a decomposition of $A_0(G,\mu)$ into a direct sum of irreducible cuspidal representations, and each component $\pi$ lift to an irreducible representation $\tpi$ of $\tG$ as functions, i.e. the restriction of functions in $V_\tpi$ contains $V_\pi$. 

\section{Excursion characters for $GL(n)$}\label{sec:excursion algebra}

As a warm-up, we will first treat the excursion characters in the case of $G=GL(n)$. We will show that a excursion character $\tnu$ is determined by its restriction to unramified Hecke subalgebra $\TN=\otimes'_{v\nmid N} H_v$, where $N$ is the fixed effective divisor. As a consequence, Lafforgue's decomposition (\ref{eqn: Lafforgue decomposition}) coincides with the classical decomposition (\ref{eqn:leveled classical decomposition}).

Suppose $\tnu: \BN \rightarrow \bar{\QQ}_l$ is a character, then by studying the pseudocharacter $\tnu(S_{n,f,(\gamma_1,\gamma_2,\ldots \gamma_n)})$, where $\gamma_i \in \Gamma$, and $f\in \mathcal{O}(\hat{G}^n//\hat{G})$, we know that $\tnu$ is determined by $\tnu(S_{1,\text{Tr},\gamma})$, where $\gamma \in \Gamma$. Since $\gamma \mapsto \tnu(S_{1,\text{Tr},\gamma})$ is continuous, $\tnu$ is determined by $\tnu(S_{1,\text{Tr},\text{Frob}_v})$, where $v$ runs through places of $F$ outside $N$. This is because $\{\text{Frob}_v \vert v\nmid N\}$ is a dense subset of $\Gamma$ by Chebotorev's Theorem.

However, by compatibility with Hecke operators, we know that $S_{1,\text{Tr},\text{Frob}_v}$ is the Hecke operator of $h_{V,v}$ corresponds to the standard representation $V$ of $\hat{G}=GL(n,\bar{\QQ}_l)$. Thus, $\tnu$ is determined by $\tnu\vert_{\TN}$.

We have the decompositoin
$$ A_0(G,K_N,\bar{\QQ}_l) = \bigoplus_{\pi} \pi^{K_N} = \bigoplus_{\lambda} A_{0,\lambda},$$
where each $\pi^{K_N}$ becomes the eigenspace $\mathcal{A}_{0,\lambda}$, for some character $\lambda$ of $T_N$. By Strong Multiplicity One theorem, we know that different $\pi^{K_N}$ corresponds to different $\lambda$.

On the other hand, we also have the decomposition
$$ \mathcal{A}_0(G,K_N,\bar{\QQ}_l)=\bigoplus_{\tnu} \mathcal{A}_{0,\tnu} = \bigoplus_{\lambda} A_{0,\lambda},$$
where $\mathcal{A}_{0,\tnu}$ is the eigenspace of character $\tnu$ of $\BN$. When restricting to $\TN$, $\mathcal{A}_{0,\tnu}$ becomes the eigenspace $\mathcal{A}_{0,\lambda}$, where $\lambda=\tnu\vert_{T_N}$, we just showed that different $\tnu$ correspond to different $\lambda$. Hence, the two decompositions (\ref{eqn: Lafforgue decomposition}) and (\ref{eqn:leveled classical decomposition}) coincide with each other.


\section{Decomposition induced from $\mathcal{H}(GL(n))$}

Following Section \ref{sec:global}, we will first study the restriction of cusp forms from $\tG(\AA)$=GL$(n,\AA)$ to $G(\AA)$=SL$(n,\AA)$ more carefully, and define an action on $ \mathcal{A}_0(G,K,\bar{\QQ}_l)$ via this restriction. Recall that we have a bijection
$$\xymatrixcolsep{5pc}\xymatrix{
   \bigoplus_{\{\tpi\}_G\in\pcusp(\tG,\tmu)_G}\mathcal{L}(\tpi)^1  \ar[r]^\res & \cusp(G,\mu)},$$
where $\tmu \in [\mu]$, from Theorem \ref{thm:main}. Note that we can always find $\tmu \in [\mu]$ of finite order. Since $\II/\Fu \cong \II^1/\Fu  \times \ZZ$ as a topological group, and $Z(\AA)/Z(F)$ is isomorphic to the closed subgroup of $n$-th roots of unity in $\II^1/\Fu$, one can lift $\mu$ to a finite order character of $\II^1/\Fu$, and define $\tmu$ to be that lift tensoring with the trivial character on $\ZZ$. We will henceforth fix such a finite order $\tmu$ that lifts $\mu$.

In the following discussion, we need to fix a specific lifting, i.e., we need to specify a cuspidal representation $\tpi$ from each $G$-equivalence class $\{\tpi\}_G$. The actions induced from GL$(n,\AA)$ and its Hecke algebra do depend on this choice, however, all choices of $\tpi$ will result in the same conclusion of Proposition \ref{prop:L-packets}, and the same decomposition in Proposition \ref{prop:decomp of induced hecke algbera}, only parameterized by different characters.

\subsection{Embedding Hecke Algebra}\label{subsec:embedding hecke}
We first study how to embed Hecke algebra on SL$(n)$ to that of GL$(n)$. We fix the ground field $k=F_v$ to be a local field, with $\pi$ a uniformizer. For $G=SL(n)$, we fix the standard maximal torus in the Borel group, $T \subset B \subset G$, and $X^\bullet( \text{resp. }X_\bullet)$  be the (resp. co)character group of $G$. Let $\Phi = \Phi^+ \cup -\Phi^+$ be the roots, $\rho=\frac1{2}\sum_{\Phi^+}\alpha$, and $P^+$ be the positive coroots, $W$ be the Weyl goup. Let $G$ also denotes the $k$-points $G(k)$, and $K=G(\mathfrak{\O})$ be the maximal compact open subgroup of $G$.
Let $\mathcal{H}(G,K)$ be the Hecke algebra of locally constant compactly supported function, bi-invariant under $K$, we normalize the measure d$x$, so that volume of $K$ is $1$. And we define everything with $\tilde{}$ to indicate the corresponding notions of $GL(n)$, with the exceptions that $\Phi^+$ and $\rho$ can indicate both for $GL(n)$ and $SL(n)$, since they have the identical root system.

We have the Cartan Decomposition:
\begin{proposition}
  The group $G$(resp. $\tG$) is disjoint union of double coset $K\lambda(\pi)K$(resp. $\tK\lambda(\pi)\tK$), where $\lambda$ runs over $P^+$.
\end{proposition}
Hence a basis of Hecke algebra $\mathcal{H}(\tG,\tK)$(resp. $\mathcal{H}(G,K)$) is the set of characteristic functions $\tc_\lambda=\text{Char}(\tK\lambda(\pi)\tK)$(resp. $c_\lambda = \text{Char}(K\lambda(\pi)K)$), where $\lambda$ runs over $P^+$.  We define the embedding to be:

\begin{align*}
\iota:\mathcal{H}(G,K)&\hookrightarrow \mathcal{H}(\tG,\tK)\\
c_\lambda &\mapsto \tc_\lambda,
\end{align*}
for all $\lambda \in P^+(\text{SL}(n))$. Since $\{c_\lambda \vert \lambda \in P^+(\text{SL}(n))\}$ form a $\overline{\mathbb{Q}_l}$-basis of Hecke algebra, we can extend this linearly to $\mathcal{H}(G,K)$. This is clearly injective, and additive, it suffices to show that it is also multiplicative.

\begin{lemma}\label{lem:doublecoset decomp}
  If $K\lambda(\pi)K= \coprod x_iK$, then $\tK\lambda(\pi)\tK= \coprod x_i\tK$, where $\lambda \in P^+(\text{SL}(n))$.
\end{lemma}
\begin{proof}
It is clear that $\bigcup x_i\tK \subset \tK\lambda(\pi)\tK$, and it is a disjoint union. Moreover, any element
\begin{align*}
  \tk \lambda(\pi) \tk'&=(kd)\lambda(\pi)\tk'\\
                       &=kd\lambda(\pi)\tk'\\
                       &=k\lambda(\pi)d\tk'\\
                       &=x_ik'd\tk' \in x_i\tK, \quad \text{for some } x_i
\end{align*}
where $\tk,\tk' \in \tK, k,k'\in K$, and $d \in \tK$ is some diagonal matrix, with the first entry equals det$(\tk)$, and $1$ elsewhere. Since $\lambda(\pi)$ is diagonal too, it commutes with $d$.
\end{proof}
Standard computations \cite{gross1998satake} show that $$\tc_\lambda \cdot \tc_\mu = n_{\lambda,\mu}(\nu)\tc_\nu,$$ where $\tilde{n}_{\lambda,\mu}(\nu)=\sharp\{(i,j):\nu(\pi)\in x_iy_j\tK\}$. For $\nu \in P^+$, Lemma \ref{lem:doublecoset decomp} implies that $\tilde{n}_{\lambda,\mu}(\nu)=n_{\lambda,\mu}(\nu)$, hence $\iota$ is a $\overline{\mathbb{Q}_l}$-algebra map.
The embedding $\iota$ also preserves Satake isomorphism: $$\mathcal{S}(f)(t)=\delta(t)^{1/2}\cdot \int_Nf(tn)dn, \quad f\in \mathcal{H}(G,K),$$ where d$n$ is the unique haar measure on $N$, such that $N\cap K$ has volume $1$, and $\delta(t)^{1/2}=q^{-<\mu,\rho>}$ for $t=\mu(\pi)$. This is because $\tG$ and $G$ has the same coroots, and hence $\rho$. It can also be checked for $f = c_\lambda$ that $\mathcal{S}(\iota(f))$ is supported on $X_\bullet(T)$, and
$$\mathcal{S}(\iota(f))\vert_{X_\bullet(T)}=\mathcal{S}(f), \quad f\in \mathcal{H}(\tG,\tK).$$ Thus, $\mathcal{S}(\iota(f))=\iota'(\mathcal{S}(f))$, where $\iota'$ is the natural inclusion of $\overline{\mathbb{Q}_l}[X^\bullet(\widehat{T})]^W \hookrightarrow \overline{\mathbb{Q}_l}[X^\bullet(\widehat{\tT})]^W$. Hence, if $V$ is a representation PGL$(n,\overline{\mathbb{Q}_l})\rightarrow \text{GL}(m,\CC)$, which lifts to a representation $\tilde{V}$ of GL$(n,\overline{\mathbb{Q}_l})$, then $\iota(h_{V,v})=h_{\tilde{V},v}$, where $h_{V,v}$ is the Hecke operator corresponds to $V$.

\subsection{Action induced from Hecke algebra}\label{subsec:induced action from gln hecke}
We can get an action of Hecke algebra of GL$(n)$ to $\cusp(G,\mu)$, via the restriction in Theorem \ref{thm:main}.

Fix $N_0$ an effective divisor, and $K_{N_0} \subset G(\O_F)$ the corresponding open compact subgroup of $G(\AA)$, where $K_v=G(\O_{F_v})$ for $v\nmid N_0$. Hence we have the decomposition of cusp forms of level $K_{N_0}$: $$\cusp(G,K_{N_0},\mu) = \bigoplus \pi^{K_{N_0}},$$ which is finite dimensional. Choose a finite basis $\{f_i\}$ of $\cusp(G,K_{N_0},\mu)$, we get a lifting $\{\tf_i\} \subset \cusp(\tG,\tmu)$, where $\tmu \in [\mu]$, and of finite order. Since $\tf_i$ is locally constant, there exists a compact open subgroup $\tK_i$ such that $\tf_i$ is right-invariant under $\tK_i$. By taking finite intersection, we can assume that $\{\tf_i\} \subset \cusp(\tG,\tK,\tmu)$, for some open compact subgroup $\tK \subset \prod_v \tG(\mathcal{O}_v)$ of $\tG(\AA)$.

We can then choose $N$, such that $N_0 \subset N$, and $\tK_v=\tG(\O_{F_v})$, for $v \nmid N$. Let $T_{N}=\otimes_{v\nmid N}H_v(G(F_v),K_v)$, and $\tT_{N}=\otimes_{v\nmid N}H_v(\tG(F_v),\tK_v)$ be the unramified commutative Hecke algebra of $G$, and $\tG$ respectively. Therefore, we can extend the action of  $T_N$ on $\cusp(G,K_{N_0},\mu)$ to $\tT_N$ via $\res$. More precisely, $\tilde{h}\cdot f = \res(\tilde{h}\cdot \tf)$. Note that $\tf \in \tpi^{\tK}$, hence for all places $v\nmid N$, we have $\tilde{h}_v\cdot f = \tlambda_v(\tilde{h}_v)f$, where $\tlambda_v$ is the Hecke character of $\tpi_v$. If $\tilde{h}_v=\iota(h_v)$, then $\tilde{h}_v\cdot f = h_v\cdot f$. This can be checked similarly by Lemma \ref{lem:doublecoset decomp}. Let $\tlambda=\otimes_{v\nmid N}\tlambda_v$, then $\lambda = \tlambda\vert_T$ is the normal character for $T$.

Suppose $\tpi = \otimes_v' \tpi_v$, by Lemma \ref{lem:multiplicity} we know that $$ \text{Res}^\tG_G \tpi = \text{Res}^\tG_G \otimes_v' \tpi_v \cong \bigoplus_{\text{each } v, i_v=1 \ldots r_v} \otimes_v' \pi_{v,i_v},$$ where Res$^{\tG_v}_{G_v} \tpi_v = \oplus_{i=1}^{r_v} \pi_{v,i}$ is multiplicity free, and $\pi_{v,1}$ is unramified ,$i_v=1$ for almost all $v$.

Similarly we have $$\res(\tpi) = \oplus \pi_k,$$ is the restriction of cusp forms, where $\pi_k$'s are irreducible cuspidal representations of $G$, and it is multiplicity free. We call $\{\pi_k \vert \res(\tpi) = \oplus \pi_k\}$ an \textbf{L-packet} of $\tpi$, denoted as $L(\tpi)$ (we'll sometimes abuse this term by referring $\res(\tpi)$ as $L$-packet too). Also, let $\overline{L(\tpi)}$ be the set of isomorphism classes of $G(\AA)$ representations in $L(\tpi)$, it is clearly a subset of the set of (classes of) irreducible components of Res$_G^\tG(\tpi)$.

The following proposition is essentially showed in \cite{gelbart1982indistinguishability}

\begin{proposition}\label{prop:L-packets}
If $\overline{L(\tpi_1)} \cap \overline{L(\tpi_2)} \neq \emptyset$, then $\overline{L(\tpi_1)} = \overline{L(\tpi_2)}$. In particular, if two irreducible cuspidal representations of $G(\AA)$ are in the same L-packet, then they have the same multiplicity in $\cusp(G,\mu)$.
\end{proposition}

Thus, it is easy to see that the L-packet $\res(\tpi)^{K_{N_0}}=\oplus \pi^{K_{N_0}}$ is contained in $\tlambda$-eigenspace $\cusp(G,K_{N_0},\mu)_\tlambda$, where $\tlambda$ is the unramified Hecke character for $\tpi$. By Multiplicity One, they are actually equal. If two L-packets are isomorphic (contains isomorphic irreducible components), then they will be contained in the same $\lambda$-eigenspace $\cusp(G,K_{N_0},\mu)_\lambda$. We conclude in the following proposition:

\begin{proposition}\label{prop:decomp of induced hecke algbera}
  The action of $\tT_N$ extends the action of $T_N$, and we have the decomposition $\cusp(G,K_{N_0},\mu) = \bigoplus_{\lambda} \cusp(G,K_{N_0},\mu)_\lambda$, where $\cusp(G,K_{N_0},\mu)_\lambda = \bigoplus_{\tlambda, \tlambda\vert_{T}=\lambda} \cusp(G,K_{N_0},\mu)_{\tlambda}$, where $\lambda$ and $\tlambda$ are characters of $T_N$ and $\tT_N$. Moreover, $\cusp(G,K_{N_0},\mu)_{\tlambda}=\res(\tpi)^{K_{N_0}}$ is an L-packet. For each $\cusp(G,K_{N_0},\mu)_{\tlambda}$, we can associate a Langlands parameter $\sigma$ to it, which is compatible with Satake Isomophism.
\end{proposition}

For $\cusp(G,K_N,\mu)_{\tlambda}=\res(\tpi)^{K_{N_0}}$, we have a Langlands parameter $\tsigma:\Gamma \longrightarrow \text{GL}(n,\bar{\QQ}_l)$ corresponding to $\tpi$, compose it with the projection to PGL$(n,\bar{\QQ}_l)$, we get the desired Langlands parameter $\sigma$. Let $V$ be an irreducible representation of PGL$(n,\bar{\QQ}_l)$, it lifts to an irreducible representation $\tilde{V}$ of GL$(n,\bar{\QQ}_l)$. By compatibility with Satake Isomorphism for GL$(n)$, we know that $h_{\tilde{V},v}$ acts on $\cusp(\tG,\tK,\tmu)_{\tsigma}(=\tpi^\tK)$, by multiplication by the scalar $\chi_{\tilde{V}}(\tsigma(\text{Frob}_v))$, where $\chi_{\tilde{V}}$ is the character of $\tilde{V}$. Since $\iota$ is compatible with Satake Isomorphism, $h_{V,v}$ acts on $\cusp(G,K,\mu)_{\sigma}=\res(\tpi^\tK)$, also by multiplication by the scalar $\chi_V(\sigma(\text{Frob}_v))=\chi_{\tilde{V}}(\tsigma(\text{Frob}_v))$.

\begin{remark}
  The definition of $T_N$ and $\tT_N$ involves $N$. Hence, in Proposition \ref{prop:decomp of induced hecke algbera}, the decomposition under the action of $T_N$ does depend on $N$. However, the decomposition under the action of $\tT_N$ does not depend on $N$, thanks to Multiplicity One.
\end{remark}

\begin{remark}
  The action induced from $\tT_N$ does depend on the choice of lifting. If we choose $\tpi_1=\tpi_2\otimes \omega$, where $\omega\in (\tG(\AA)/\tG(F)G(\AA)Z(\AA))^D$, then $\res(\tpi_1)=\res(\tpi_2)$, but $\tT_N$ acts on it as $\tlambda_i$ if we lift it to $\tpi_i$, where $\tlambda_i$ are the normal Hecke character for $\tpi_i$. By Strong Multiplicity One, $\tlambda_1\neq \tlambda_2$. Again, the decomposition in Proposition \ref{prop:decomp of induced hecke algbera} doesn't depend on lifting, different liftings will only result in different characters $\tlambda$ parameterizing the same components.
\end{remark}

\section{Excursion character for SL(n)}

\subsection{Decomposition coincide} \label{subsec: coincide}
In Subsection \ref{subsec:induced action from gln hecke}, we defined an induced action of unramified Hecke algebra $\tT_N$ of GL$(n)$ to $\cusp(G, K_{N_0}, \mu)$. By the discussion in Section \ref{sec:excursion algebra}, we know that the action of $\tBN$ on $\cusp(\tG, \tmu)$ uniquely extends that of $\TN$. We can similarly define the induced action of $\tBN$ on $\cusp(G, K_{N_0}, \mu)$, this will result in the same decomposition in Proposition \ref{prop:decomp of induced hecke algbera}, but parameterized by characters $\tnu$ of $\tBN$.

Then the decomposition in Proposition \ref{prop:decomp of induced hecke algbera} can be written as: $$\cusp(G,K_{N_0},\mu) = \bigoplus_{\tnu} \cusp(G,K_{N_0},\mu)_\tnu,$$ where $\cusp(G,K_{N_0},\mu)_\tnu = \cusp(G,K_{N_0},\mu)_\tlambda=\res(\tpi)^{K_{N_0}}$, for the unique character $\tnu$ of $\tBN$, such that $\tlambda=\tnu\vert_{\tT}$, and $\tpi$ corresponds to $\tlambda$. On the other hand, we have the actual decomposition given in Lafforgue's paper for SL$(n)$. $$\cusp(G,K_{N_0},\mu) = \bigoplus_{\nu} \cusp(G,K_{N_0},\mu)_\nu,$$ where $\nu$ is the character of $\BN = \mathcal{B}_N(SL(n))$. There is an obvious way that $\BN$ can be embedded into $\tBN$, if we get $\nu$ by restricting $\tnu$ to $\BN$, then the corresponding Langlands parameter $\sigma$ is the composition of $\tilde{\sigma}$ with the projection $GL(n) \rightarrow PGL(n)$.

By Proposition 12.5 of \cite{lafforgue2012chtoucas}, we have $\cusp(G,K_{N_0},\mu)_\tnu \subset \cusp(G,K_{N_0},\mu)_\nu$. We know that the equality holds $\cusp(G,K_{N_0},\mu)_\tnu = \cusp(G,K_{N_0},\mu)_\nu$, because if $\tnu_1 \vert_{\BN} = \tnu_2 \vert_{\BN}$, then the Langlands parameter $\tilde{\sigma}_1$ and $\tilde{\sigma}_2$ are equal after composing with the projection $GL(n) \rightarrow PGL(n)$, it is easy to verify that $\tpi_1$ and $\tpi_2$ are hence $G$-equivalent, which implies $\tpi_1 = \tpi_2$, and $\tnu_1 = \tnu_2$, since we only choose one representative from a $G$-equivalence class in the decomposition of Theorem\ref{thm:main}. Meanwhile, since $\res(\tpi)$ is multiplicity free. The multiplicity of $SL(n)$ can be bounded by the number of $L$-packets $\res(\tpi)^{K_{N_0}}$, in the $\lambda$-eigenspace, this will be the subject of the next subsection.

\subsection{Higher Multiplicities of $SL(n)$}
Using the decomposition in Proposition \ref{prop:decomp of induced hecke algbera}, and its coincidence with the Lafforgue's decomposition, we may find an upper bound for the multiplicities of SL$(n)$, $M(\pi_0) = \sharp \{\pi \subset \cusp(G,\QQl) \vert \pi \cong \pi_0\}$. Since $\pi_0$ is countable dimension(countable restricted tensor product of countable dimension representations), and $\overline{\QQ_l}$ is algebraically closed and uncountable, by Schur's Lemma, $\pi_0$ has a central character $\mu$. Let $\pi_1 \cong \pi_2$ be two irreducible cuspidal representations in $\cusp(G,\mu)$ for that $\mu$. In addition, there exists $N_0$, such that $\pi_i^{K_{N_0}} \neq 0$, and both $\pi_i^{K_{N_0}}$ appear in $\cusp(G,K_{N_0},\mu)_\lambda$, for some character $\lambda$ of an unramified Hecke algebra $T_N$, hence the multiplicity of SL$(n)$ is bounded by the multiplicity in $\cusp(G,K_N,\mu)_\lambda$. By Proposition \ref{prop:decomp of induced hecke algbera}, this is bounded by the cardinality of $\{\tlambda \in \text{Hom}(\tT_N,\overline{\QQ_l}) \vert \tlambda\vert_{T_N}=\lambda, \text{ and is realized as a cuspidal representation in } \cusp(\tG, \tmu)\}/G-\text{equivalence}$.

By the work of L. Lafforgue \cite{lafforgue2002chtoucas}, we know that if $\tlambda$ comes from a cuspidal representation, it is attached with a Langlands parameter, which is the same thing as a pseudocharacter $\tTheta(m,f,(\gamma_i))$. By Chebotarev Density's theorem, we know that $\tTheta(1,Tr_V,\gamma)$ is given by $\lambda$, where $V$ is any irreducible representation of $GL(n)$ that factors through PGL$(n)$. We also know that $\tTheta(1,Det,\gamma)$ is given by $\tmu$, by Langlands Correspondance for GL$(n)$. We summarize:
\begin{proposition}\label{prop:multiplcity of SLn}
  The number of isomorphic irreducible components of $\cusp(G,\overline{\QQ_l})$ (which corresponds to a character $\lambda$ of some unramified Hecke algebra of $SL(n)$), is bounded above by the number of $G$-equivalent classes of pseudocharacters $\tTheta(m,f,(\gamma_i))$ of GL$(n)$, such that $\tTheta(1,Tr_V,\gamma)$ (where $V$ is any representation of $GL(n)$ that factors through $PGL(n)$) is given by $\lambda$, and $\tTheta(GL(n))(1,Det,\gamma)$ is given by $\tmu$.
\end{proposition}

By abuse of the notation, we say two pseudocharacters are $G$-equivalent, if their corresponding cuspidal representations are. We know that the pseudocharacters of $GL(n)$ is determined by $\tTheta(1,Tr,\gamma)$, where $\gamma$ runs over $\{Frob_v \vert v\nmid N\}$.

In Proposition \ref{prop:multiplcity of SLn}, we can take $V = (Std)^{\otimes n} \otimes \wedge^n Std$, and we see that $Tr_V = \frac{Tr^n}{Det}$, hence $\tTheta(1,Tr,\gamma)^n$, where $\gamma$ runs over $\{Frob_v \vert v\nmid N\}$, are given by $\lambda$ and $\tmu$. In another word, $\lambda$ and $\tmu$ in Proposition \ref{prop:multiplcity of SLn} determine such extensions up to $n$-th roots of unity.

\bibliographystyle{alpha}
\bibliography{note}
\end{document}